\documentclass[11pt,a4paper]{article}
\usepackage{a4wide}

\input{macros}

\numberwithin{equation}{section}

\title{Numerical schemes for a class of nonlocal conservation laws:\newline a general approach}
\author{%
  Jan Friedrich\footnotemark[1],\
  Sanjibanee Sudha\footnotemark[2]\ \ and Samala Rathan\footnotemark[2] 
}

\renewcommand{\rv}[1]{\textcolor{black}{#1}}
\begin{document}

\footnotetext[1]{RWTH Aachen University, Institute of Applied Mathematics, 52064 Aachen, Germany (friedrich@igpm.rwth-aachen.de)}
\footnotetext[2]{Department of Humanities and Sciences, Indian Institute of Petroleum and Energy, Visakhapatnam, Andhra Pradesh, India-530003 (\{sudhamath21,rathans.math\}@iipe.ac.in)}
\maketitle

\begin{abstract}
In this work we present a rather general approach to approximate the solutions of nonlocal conservation laws.
In a first step, we approximate the nonlocal term with an appropriate quadrature rule applied to the spatial discretization.
Then, we apply a numerical flux function on the reduced problem.
We present explicit conditions which such a numerical flux function needs to fulfill.
These conditions guarantee the convergence to the weak entropy solution of the considered model class.
Numerical examples validate our theoretical results and demonstrate that the approach can be applied to other nonlocal problems.
\end{abstract}
\medskip

\section{Introduction}
Scalar conservation laws with a nonlocal transport term appear in a variety of applications and can be formally summarized by
\begin{align}\label{eq:generalnonloc}
    \partial_t \rho +\partial_x F(t,x,\rho,R)=0\quad (t,x)\in\mathbb{R}^+\times\mathbb{R},
\end{align}
where $\rho$ is the state variable, $t$ the time, $x$ the space variable, $F$ a flux function and $R$ an integral evaluation over the space.
Typically, $R$ is a convolution involving the state variable $\rho$ over a possible compact space.
Such nonlocal terms can be found in various physical applications such as supply chains~\cite{goettlich2010supplychains}, sedimentation~\cite{betancourt2011nonlocal,burger2022hilliges}, conveyor belts~\cite{gottlich2014modeling, rossi2020well}, crowd motion~\cite{colombo2012class}, and traffic flow~\cite{BlandinGoatin2016, chiarello2018global, chiarello2019multiclass, friedrich2020onetoone, friedrich2020nonlocal, friedrich2018godunov, goatin2016well,keimer2018bounded}.
Thereby, these equations are studied theoretically for existence and uniqueness \cite{aggarwal2015nonlocal,amorim2015numerical, KeimerPflug2017, keimer2018multi}.
The challenge here is that classical approaches for local conservation laws cannot be applied\rv{, e.g. due to the nonlocal term}, no Riemann-solvers are currently available.
But one can use numerical schemes \cite{aggarwal2015nonlocal,amorim2015numerical} and Kruzkov's doubling of variables technique \cite{kruvzkov1970first} or fixed point approaches based on the method of characteristics \cite{KeimerPflug2017, keimer2018multi,keimer2018bounded} to prove uniqueness and existence of solutions.
Solutions must be understood in the weak sense by removing the strong regularity conditions on the solution as discontinuities may be present.
For specific flux functions, i.e. of the form $F(t,x,\rho,R)=\rho v(R)$ with a suitable nonlocal term $R$ and velocity function $v$, weak solutions are already unique \cite{KeimerPflug2017}, but for more general flux functions entropy conditions are used to single out the physically correct solutions \cite{aggarwal2015nonlocal,amorim2015numerical}.
Nevertheless, in the aforementioned works and in all the aforementioned applications, mostly specific forms of the involved flux functions and nonlocal terms are assumed.
In its whole generality the well- or ill-posedness \eqref{eq:generalnonloc} is not proved yet.
\par
From a numerical point of view the situation is very similar.
In most of the applications tailored Lax-Friedrichs-type numerical schemes have been used as a tool to prove the existence of solutions.
However, it is well known that Lax-Friedrichs type numerical schemes produce diffusive behavior.
To avoid such a problem,  a Godunov-type scheme for a specific traffic flow model has been introduced in \cite{friedrich2018godunov}.
Furthermore, higher-order schemes have been considered~\cite{ChalonsGoatinVillada2018,friedrich2019maximum}, but they rely on already derived numerical flux functions.
\rv{Numerical flux functions can also be derived by using Lagrangian-remap schemes as in \cite{abreu2022lagrangian,chiarello2020lagrangian}.
So far, these approaches are tailored to specific flux functions\footnote{We note that \cite{abreu2022lagrangian} deals with a similar flux function as we do in the following. However, for the rigorous proof of convergence they consider a simpler version of the flux.} and nonlocal terms, too.}
An alternative is to use numerical schemes without dissipation which are based on the method of characteristics \cite[Section 5]{KEIMER2023}, although the flux function must allow for a unique weak solution, i.e. be of the form $F(t,x,\rho,R)=\rho v(R)$.

Let us compare this situation with a scalar one-dimensional (local) hyperbolic conservation law of the form
\begin{equation}
\begin{aligned}
     \label{eq:scalar conservation law form}
    &\rho_t+f(\rho)_x=0,\,\,\,\,\, (t,x)\in\mathbb{R}^+\times \mathbb{R}, \\
    &\rv{\rho(0,x)}=\rho_{0}(x),\,\, x \in \mathbb{R}.
    \end{aligned}
\end{equation}
Here, not only the existence and uniqueness theory is well understood, but also the construction of numerical solutions for local conservation laws.
For more details, we refer to \cite{eymard2000finitevolume,GodlewskiRaviart,holden2015front,leveque2002finite,leveque1992numerical, thomas2013numerical}.
In particular, to solve the equation (\ref{eq:scalar conservation law form}) numerically,  one can use the most familiar three-point schemes in conservation form, i.e.,
\begin{equation}
    \begin{aligned}
       \label{eq:scalar conservative three-point} 
       \rho_{j}^{n+1}=\rho_j^n-\dfrac{\Delta t}{\Delta x}\bigg(G(\rho_j^n,\rho_{j+1}^n)-G(\rho_{j-1}^n,\rho_j^n)\bigg),
    \end{aligned}
\end{equation} 
where $\rho_j^n \approx \rho( t^n,x_j),$  numerical flux $G$ is Lipschitz continuous function and is consistent with the physical flux $f=f(u),$ i.e., $G(u,u)=f(u)$. 
If the numerical flux $G$ is monotone, then the scheme \eqref{eq:scalar conservative three-point} is total variation diminishing (TVD) and satisfies the maximum principle \cite{eymard2000finitevolume, crandall1980monotone, leveque2002finite}.
By choosing a suitable monotone numerical flux function $G,$ the numerical solution converges to a weak entropy solution of the conservation law \eqref{eq:scalar conservation law form}. 
In particular, there are several choices of monotone numerical flux functions, so one can choose the flux functions that are the best suited for each individual modeling problem.

As aforementioned for nonlocal conservation laws, such a general class of numerical schemes does not currently exist. 
One difficulty is that in general the schemes are no longer monotone due to the nonlocal term, see e.g. \cite{BlandinGoatin2016,huang2023asymptotically}.
Recently, a first attempt to consider a more general class of numerical schemes is considered in \cite{huang2023asymptotically}, however they consider very specific assumptions on the flux function, see Remark \ref{rem:comparison} for a more detailed discussion.
In this work we present a different approach to construct a general class of numerical schemes than in \cite{huang2023asymptotically}.
In particular, we consider a rather general nonlocal flux function to further fill the gap in the current literature.
The derived schemes will thereby share some similarities with the class of monotone schemes for local conservation laws.

We organize this paper as follows: In Section \ref{sec:model}, we recall the main results of \cite{chiarello2018global} on the existence and uniqueness of entropy solutions for a class of nonlocal conservation laws.
Then in Section \ref{sec:numclass} we present an approach to construct numerical schemes for the considered model class.
This includes necessary assumptions on the numerical flux function.
Furthermore, we present several numerical schemes belonging to this class, e.g. a newly derived Engquist-Osher type scheme.
Section \ref{sec:result} presents our main result on the convergence of the calls of numerical schemes.
The proof is given in the remaining section and we discuss some important properties of the proposed schemes, such as a maximum principle, bounded variation (BV) and time-continuity estimates.
In Section \ref{sec:numerics}, we present numerical examples and compare the $L^1$-accuracy of some exemplary schemes belonging to the proposed class.
Furthermore, we calculate the numerical convergence rates and demonstrate how the approach can be adapted to other classes of nonlocal conservation laws.
Finally, concluding remarks are given in Section \ref{sec:conclusion}.

\section{Modeling class}\label{sec:model}
The goal of this work is to present a general approach to construct numerical schemes for nonlocal conservation laws. Therefore, we rely on an already well-posed class of nonlocal conservation laws.\\
Here, we briefly recall the model class and the main result of \cite{chiarello2018global}.
In \cite{chiarello2018global}, the following nonlocal conservation law is studied:
\begin{align}\label{eq:ND}
&\partial_t \rho + \partial_x \left( g(\rho) V(t,x)\right)=0,  \\
&V(t,x):=v\left(\int_x^{x+\ndt} \omega_\eta(y-x)\rho(t,y)dy \right)\label{eq:ND:conv}.
\end{align}
This model considers a mean downstream density and is commonly used in various traffic flow models.
For example, it can be used to describe a nonlocal LWR model or Arrhenius-type look-ahead dynamics.
To prove well-posedness, the authors of \cite{chiarello2018global} need the following hypotheses on the functions involved:
\begin{assumption}\label{ass:model}
We assume the following hypotheses: 
\begin{enumerate}
    \item $\rho(0,x)=\rho_0(x) \in \BV(\mathbb{R},I), \, \, I=\rv{[\rho_{\min},\rho_{\max}]}\subseteq\mathbb{R^+},$
    \item $g\in C^1(I;\mathbb{R}^+),$
    \item $v\in C^2(I;\mathbb{R}^+),\quad v'\leq 0,$
    \item  $\omega_{\eta}\in C^1([0,\eta]),\mathbb{R^+})$\,\,\,\, \text{with}\,\,\, $\omega^{\prime}\leq0,$\,\,\,$\int_{0}^{\eta}\omega_{\eta}(x)dx=1$ \,\,\,$\forall\eta>0,$\,\,\,\,$\underset{\eta\to\infty}{\lim} \omega_{\eta}(0)=0,$
\end{enumerate}
\rv{where $\rho_{\min}:=\underset{\mathbb{R}}{\min}(\rho_0)$ and $\rho_{\max}:=\underset{\mathbb{R}}{\max}(\rho_0)$.}
\end{assumption}
Note that, unlike \cite{chiarello2018global}, we consider  normalized kernel functions $\wt$, which means that the $L^1$-norm is one and the interval for the density is $[0,1]$.
Due to the presence of the function $g$, which can be nonlinear, the results of e.g. \cite{KeimerPflug2017} cannot be applied.
These results hold only for the special case where $g$ is the identity.
Hence, instead of weak solutions the authors of \cite{chiarello2018global} consider, as usual for local conservation laws, weak entropy solutions of the problem \eqref{eq:ND}.
These are intended in the following sense, see \cite[Definition 1.1]{chiarello2018global}:
\begin{definition}\label{def:weakentropy} A function $\rho\in(L^1\cap L^{\infty}\cap \BV)(\mathbb{R}^+\times\mathbb{R};\mathbb{R})$ is a weak entropy solution of \eqref{eq:ND} with initial condition $\rho(0,x)=\rho_0(x)$, if
\begin{align}
    \label{eq:entropy}
    \int_0^{+\infty}\int_{-\infty}^{+\infty}\bigg(|\rho-k|\phi_t+|\rho-k|V\phi_x-\sign(\rho-k)kV_x\phi\bigg)(t,x)dxdt+\int_{-\infty}^{+\infty}|\rho_0(x)-k|\phi(0,x)dx\geq0
\end{align}
for all $\phi\in C_{c}^1(\mathbb{R}^2;\mathbb{R})\, \text{and}\, k\in\mathbb{R}.$
\end{definition}
Now, we can state the main result of \cite{chiarello2018global}:

\begin{theorem}[compare to Theorem 1.2 of \cite{chiarello2018global}]
Let the Assumptions \ref{ass:model} hold, then the Cauchy problem 
\begin{eqnarray*}
\begin{cases}
&\partial_t\rho(t,x)+\partial_x\bigg(g(\rho(t,x))V(t,x)\bigg)=0,\, \, x\in\mathbb{R},t>0,\\
 &\rho(0,x)=\rho_0(x), \, \,\, x\in\mathbb{R}
\end{cases}
\end{eqnarray*}
with $V(t,x)$ defined as in \eqref{eq:ND:conv} admits a unique weak entropy solution in the sense of Definition \ref{def:weakentropy}, such that
$$\rho_{\min}=\underset{\mathbb{R}}{\min}(\rho_0)\leq\rho(t,x)\leq\underset{\mathbb{R}}{\max}(\rho_0)=\rho_{\max},\,\,\,   \text{for a.e.}\,\,\,\, x\in\mathbb{R}, t>0.$$
\end{theorem}
In \cite{chiarello2018global} a Lax-Friedrichs-type numerical scheme is used as a tool to prove the existence of solutions.
Our main goal now is to derive a complete class of numerical schemes, which provides the flexibility to choose or develop numerical schemes that are the best suited for each application.
In particular, the numerical schemes may be less diffusive than a Lax-Friedrichs-type numerical scheme.

\section{A general class of numerical schemes for nonlocal conservation laws}\label{sec:numclass}
In the following we will introduce a more general class of numerical schemes for the problem \eqref{eq:ND}.
From now on, unless otherwise stated, $\norm{\cdot}$ defines for simplicity the $L^\infty$ norm over the underlying space, e.g., $\norm{v}:=\norm{v}_{L^\infty([\rho_{\min},\ \rho_{\max}])}$.

We discretize space and time by an equidistant grid, where $\Dx$ is the step size in space and $\Dt$ is the step size in time.
Hence, $t^n=n\Dt$ with $n\in\N$ describes the time grid and $x_j=j\Dx,\ j\in\Z$ the cell centers of the space grid with the cell interfaces $x_\jmh$ and $x_\jph$.
To construct a finite volume approximation $\dr$ such that $\dr(t,x)=\rho_j^n$ for $(t,x)\in[t^n,t^{n+1})\times [x_\jmh,x_\jph)$, we approximate the initial data by
\begin{align*}
    \rho_j^0=\frac{1}{\Dx} \int_{x_\jmh}^{x_\jph} \rho_0(x) dx,\quad j\in\Z.
\end{align*}
The scheme is then given by 
\begin{align}\label{eq:scheme}
    \rho_j^{n+1}=\rho_j^n-\lambda \left(F_\jph^n(\rho_j^n,\rho_{j+1}^n)-F_\jmh^n(\rho_{j-1}^n,\rho_{j}^n)\right)\quad\text{with}\quad \lambda:=\frac{\Dt}{\Dx},
\end{align}
with a suitable numerical flux $F_\jph^n(a,b)$ which needs to be determined.\\

Therefore, we want to consider a rather general approach to derive the numerical flux function that approximates the flux passing through the cell interface $x_\jph$.
Our main idea is as follows:
We consider the piecewise constant reconstruction of the density $\dr$.
Then, the nonlocal term \eqref{eq:ND:conv} evaluated at the cell boundary $x_\jph$ and time $t^n$ is a fixed value which we call $V_j^n$, see Figure \ref{fig:discret}.
Hence, we can apply an appropriate numerical flux function similar to those known for local conservation laws to approximate the reduced problem with a fixed nonlocal term, i.e. the flux reduces to 
\begin{equation}\label{eq:reducedflux}
F(t^n,x_\jph,\rho)\approx g(\rho)V_j^n.    
\end{equation}

\begin{figure}
    \centering
	\newcommand\checkindexnew[1]{
		\pgfmathsetmacro{\var}{#1}
		\pgfmathparse{ifthenelse(\var==0, "",ifthenelse(\var>0, "+#1","#1"))} \pgfmathresult}%

\begin{tikzpicture}
	
	\pgfmathsetmacro{\celllength}{2.4}
	\pgfmathsetmacro{\cellheight}{1.5}
	\pgfmathsetmacro{\notationposshift}{0.21}

	\def\plotfunc(#1,#2){0.25*(#2)*sin(50*((#1)-2*\celllength))+0.5*(#2)}

	\foreach \j in {-1, ..., 4}{
	   	\pgfmathsetmacro{\currleftbound}{{(\j-0.5)*\celllength}}
	   	\pgfmathsetmacro{\currrightbound}{{(\j+0.5)*\celllength}}
		\pgfmathsetmacro{\currcenter}{{\j*\celllength}}
		\pgfmathsetmacro{\variablefa}{\plotfunc(\currleftbound,\cellheight)}
		\pgfmathsetmacro{\variablefab}{\plotfunc(\currcenter,\cellheight)}
		\pgfmathsetmacro{\variablefb}{\plotfunc(\currrightbound,\cellheight)}
		\pgfmathsetmacro{\curraverage}{
				{((\currrightbound-\currleftbound)/6 * (\variablefa + 4 * \variablefab + \variablefb))
					/ \celllength}}
		
		\if \j3{
		\filldraw[gray!30] (\currleftbound,0) -- (\currleftbound, \curraverage) -- (\currrightbound, \curraverage) -- (\currrightbound, 0) -- cycle;}
		\else
		\if \j2{
		\fill[gray!30] (\currleftbound,0) -- (\currleftbound, \curraverage) -- (\currrightbound, \curraverage) -- (\currrightbound, 0) -- cycle;}	
\else
		\if \j1{
		\fill[gray!30] (\currleftbound,0) -- (\currleftbound, \curraverage) -- (\currrightbound, \curraverage) -- (\currrightbound, 0) -- cycle;}
		\else
		\fi
		\fi		
		
		\fi
		\draw[thick] (\currleftbound, \curraverage) -- (\currrightbound, \curraverage);		

	}
	
	\foreach \j in { -1, ...,4}{
	   \pgfmathsetmacro{\ifpos}{{(\j-0.5)*\celllength}}
	   \draw[thick] (\ifpos , 0) -- (\ifpos , \cellheight);
	   \if \j5
	  {}
	  \else
\def \currindex {\checkindexnew{\j}} 
	   	\node at ({\ifpos+(\celllength*0.5)},\cellheight*0.9) 
	   				{\small $\rho_{j\currindex}$};
	   				\fi
	   				
	}

	\foreach \j in { -1, ..., 4}{
	   	\draw (\j*\celllength, -0.1) -- (\j*\celllength, 0.1);
		\def \currindex {\checkindexnew{\j}} 
	   	\node at (\j*\celllength, -0.4) {$x_{j\currindex}$};
	}
	
	\draw[thick,domain=-1.5*\celllength:4.5*\celllength,smooth,variable=\x] 
			plot ({\x},{\plotfunc(\x, \cellheight)});
		
\draw[->] (3.5*\celllength , -0.7) -- (3.5*\celllength , -0.1);
	   	\node at (3.5*\celllength, -1) {$x_{j+7/2}=x_{j+\frac{1}{2}}+\ndt=x_\jph+3\Dx$};
	   	
\draw[->] (0.5*\celllength , -0.7) -- (0.5*\celllength , -0.1);
	   	\node at (0.5*\celllength, -1) {$x_{j+\frac{1}{2}}$};	
	\draw[thick] (-1.6*\celllength,0) -- (4.6*\celllength, 0);

\draw[->] (0.5*\celllength , 1.3*\cellheight) -- (0.5*\celllength , 1.1*\cellheight);
	   	\node at (0.5*\celllength, 1.5*\cellheight) {$F_{j+\frac{1}{2}}$};
  \draw[thick,decorate, decoration={brace}, yshift=0.5ex] (0.5*\celllength , \cellheight) -- (3.5*\celllength , \cellheight) node[midway,above]{\textcolor{gray}{$V_j^n$}};

 \end{tikzpicture}
    \caption{Space discretization and approximation of the nonlocal term for $\ndt=3\Dx$ with the densities used to calculated $V_j^n$ highlighted in gray.}
    \label{fig:discret}
\end{figure}
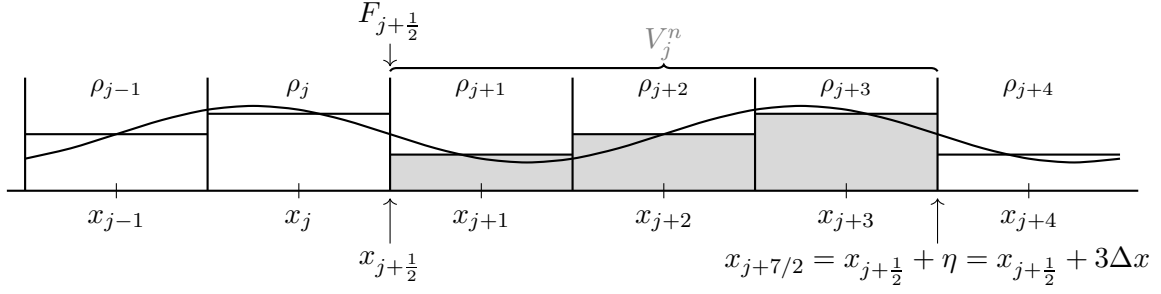
Now let us look more closely at the approximation of the nonlocal term and the requirements for a suitable numerical flux function.
\begin{definition}[Approximation of the nonlocal term]\label{def:NLtermdiscrete}
    We approximate the nonlocal term  \eqref{eq:ND:conv} by
\begin{equation}\label{eq:approximatedNLterm}
\begin{aligned}
    &V(t^n,x_\jph)\approx v\left(\int_{x_\jph}^{x_\jph+\ndt} \wt(y-x)\dr(t^n,y)dy \right)= v\left(\sum_{k=0}^{\Ne-1} \gamma_k \rho_{j+k+1}^n\right)=:V_j^n\\
&\text{with}\quad 
    \gamma_k=\int_{k\Dx}^{(k+1)\Dx} \wt (x) dx\quad \text{and}\quad \Ne:=\lfloor \ndt/\Dx\rfloor.
\end{aligned}
\end{equation}
The weights $\gamma_k$ must be computed exactly, e.g. by an appropriate quadrature rule.
\end{definition}
An illustration of the nonlocal term for $\ndt=3\Dx$ can be seen in Figure \ref{fig:discret}.
To ensure the positivity of the nonlocal term it is important that the weights are computed without any numerical error.
For the reduced flux \eqref{eq:reducedflux} $V_j^n$ needs to be considered as a constant and its dependence on $\rho_\jpo^n,\ldots,\rho_{j+\Ne}^n$ will be neglected in the following.
Then a suitable numerical flux function $F^n_\jph(a,b)$ must satisfy the following definition:

\begin{definition}[Numerical flux function]\label{def:numscheme}
The numerical flux function $F^n_\jph$ satisfies the following conditions:
\begin{enumerate}[label=(\roman*)]
\item The flux function can be written as
\[F^n_\jph(a,b)=G(a,b)V_j^n\]
with $G$ only depending on $g,\ \rho_{\min},\ \rho_{\max}$. 
\item Consistency:
\[F^n_\jph(\rho,\rho)=g(\rho)V_j^n,\quad\text{or equivalently },\quad G(\rho,\rho)=g(\rho)\,\,\, \forall \rho\in [\rho_{\min},\rho_{\max}].\]
\item The map $(a,b)\mapsto F^n_\jph(a,b)=\rv{G}(a,b)V_j^n$ from $[\rho_{\min},\rho_{\max}]^2$ to $\R$ is nondecreasing with respect to $a$ and nonincreasing with respect to $b$.
\item $F^n_\jph(a,b)$ satisfies the following sort of weak Lipschitz continuity:
\begin{align}\label{eq:weakLip}
    \abs{F^n_\jph(a,b)-F^n_\jph(b,b)}\leq L_1 \abs{a-b}\quad\text{and}\quad\abs{F^n_\jph(a,b)-F^n_\jph(a,a)}\leq L_2\abs{a-b}
\end{align}
for $(a,b)\in [\rho_{\min},\rho_{\max}]^2$ and every $V_j^n\in [-\norm{v},\norm{v}]$. We will call $L_1>0$ and $L_2>0$ the Lipschitz constants.
\end{enumerate}
\end{definition}

\begin{remark}
Note that the assumptions imply that $F^n_\jph$ and $G$ are bounded and that $G$ satisfies a statement similar to \eqref{eq:weakLip}.
The latter follows directly from $F^n_\jph$ satisfying \eqref{eq:weakLip} for every $V_j^n\in [-\norm{v},\norm{v}]$ and the bound can be obtained by
\[F^n_\jph(a,b)=F^n_\jph(a,b)-F^n_\jph(b,b)+F^n_\jph(b,b)\leq g(b)V_j^n+L_1\abs{a-b}\leq \norm{g}\norm{v}+L_1(\rho_{\max}-\rho_{\min}).\]
\end{remark}

The assumptions on the flux function after approximating the nonlocal term are very similar to those for a monotone flux scheme of local conservation laws, cf. \cite[Definition 5.6]{eymard2000finitevolume}.
The only additional assumption is that we assume the numerical flux to be decomposable as in (i).
Since the flux function of \eqref{eq:ND} is decomposable into a nonlinear term and nonlocal term, i.e. $F(t,x,\rho)=g(\rho)V(t,x)$, this is a rather natural assumption.
The other assumptions (ii)-(iv) are also quite common.
Assumption (ii) is necessary for the consistency with the flux of the modeling equation, while (iii) gives our scheme a monotone like behavior.
Although the complete scheme is not monotone in every argument, as here the dependence of the nonlocal term on $\rho_\jpo,\dots,\rho_{j+\Ne}$ is not considered.
Finally, (iv) is a slightly weaker form of the usual assumed local Lipschitz continuity.
One could also assume that the flux function is locally Lipschitz continuous from $\R^2$ to $\R$ as typical for monotone flux schemes,  cf. \cite[Definition 5.6]{eymard2000finitevolume}, but during the convergence analysis in Section \ref{sec:result} it will turn out that the weak Lipschitz continuity \eqref{eq:weakLip} is sufficient to guarantee convergence.

\begin{remark}[Generalization of the flux]
The above strategy can be applied to other classes of nonlocal conservation laws that have a similar structure.
For simplicity, we will perform the proofs only for the equation \eqref{eq:ND}.
We will discuss this further in the numerical examples in Section \ref{sec:numsedimentation}.
Note that the Definition \ref{def:numscheme} also covers the case of $V_j^n<0$, although this does not occur in the model \eqref{eq:ND}.
\end{remark}

Let us present some examples for the equation \eqref{eq:ND} which are inspired by classical monotone schemes.
In Section \ref{sec:result}, we will see that $V_j^n\geq 0$ holds due to Definition \ref{def:NLtermdiscrete} and the Assumptions \ref{ass:model}.
Hence, we will already use $V_j^n\geq 0$ to simplify the schemes in the following.
Whenever we use this property, we denote it with $(*)$.
\begin{itemize}
\item\textbf{(local) Lax-Friedrichs-type scheme:} For a diffusion parameter $\alpha>0$ and inspired by the classical Lax-Friedrichs scheme \cite{lax1954scheme}, we have
\begin{align}\label{eq:LxFmod}
    F^n_\jph(\rho^n_j,\rho^n_{j+1})=\frac{V_j^n}{2}\left( g(\rho^n_j)+g(\rho^n_{j+1})+\alpha\left(\rho^n_j-\rho^n_{j+1}\right)\right).
\end{align}
To satisfy the Definition \ref{def:numscheme} $\alpha\geq \norm{g'}$ needs to hold.\\
Note that this scheme differs to the usual choice of the nonlocal Lax-Friedrichs-type scheme as used e.g. in \cite{chiarello2018global}.
Here, we consider $g(\rho^n_{j+1})V_j^n$ instead of $g(\rho^n_{j+1})V_{j+1}^n$.
\item\textbf{Godunov-type scheme:} Inspired by the Godunov scheme introduced in \cite{godunov1959} we deduce the following flux function
\begin{align}\notag
    F^n_\jph(\rho_j^n,\rho^n_{j+1})=&\begin{cases}
    \min_{[\rho^n_j,\,\rho^n_\jpo]}g(\rho)V_j,&\rho^n_j\leq\rho^n_\jpo,\\
    \max_{[\rho^n_\jpo,\,\rho^n_j]}g(\rho)V_j,&\rho^n_j>\rho^n_\jpo,
    \end{cases}\\
    \label{eq:Godunov}
    \stackrel{(*)}{=}&\,V_j^n\begin{cases}
    \min_{[\rho^n_j,\,\rho^n_\jpo]}g(\rho),&\rho^n_j\leq\rho^n_\jpo,\\
    \max_{[\rho^n_\jpo,\,\rho^n_j]}g(\rho),&\rho^n_j>\rho^n_\jpo.
	\end{cases}
\end{align}
\item\textbf{Engquist-Osher-type scheme:} Further, we consider an Engquist-Osher-type scheme following the ideas of \cite{engquist1981one}:
\begin{align}\notag
    F^n_\jph(\rho^n_j,\rho^n_{j+1})=&\frac{1}{2}\left( g(\rho_j^n)V_j^n+g(\rho_{j+1}^n)V_j^n-\int_{\rho_j^n}^{\rho_{j+1}^n}\abs{g'(\rho)V_j^n}d\rho\right)\\
    \label{eq:EOflux}
    \stackrel{(*)}{=}&\,\frac{V_j^n}{2}\left( g(\rho_j^n)+g(\rho_{j+1}^n)-\int_{\rho_j^n}^{\rho_{j+1}^n}\abs{g'(\rho)}d\rho\right).
\end{align}
Note that in case of a strictly concave function $g$, i.e. $g''(\rho)>0$ and $g'(c)=0$ for $c\in [\rho_{\min},\rho_{\max}]$, the flux simplifies to
\[F^n_\jph(\rho_j^n,\rho_{j+1}^n) =\rv{V_j^n}\left(g(\min\{\rho_j^n,c\})+g(\max\{\rho_{j+1}^n,c\})-g(c)\right),
    \]
compare \cite[p. \rv{229}]{GodlewskiRaviart}.
\end{itemize}
All these schemes are inspired by already well studied schemes for local conservation laws, in particular in the local case they all belong to the class of monotone schemes, see \cite{eymard2000finitevolume}.
Hence, they also satisfy Definition \ref{def:numscheme}.
Note that the Godunov-type scheme has already been considered for similar problems in \cite{chiarello2019multiclass,colombo2019role, friedrich2018godunov}.
In contrast, the Lax-Friedrichs-type scheme is slightly different from the usual ones used in the literature, e.g., \cite{BlandinGoatin2016,chiarello2018global, huang2023asymptotically}.
To the best of our knowledge this is the first time that an Engquist-Osher-type scheme is proposed for nonlocal conservation laws.
Furthermore, an Upwind-type scheme does not in general satisfy Definition \ref{def:numscheme} for a nonlinear $g$, since it is not monotone.
As in the local case it may converge to the wrong solution.
Nevertheless, if $g$ is linear, e.g. $g(\rho)=\rho$, also an Upwind-type scheme would converge to the correct solution.
In particular, in this case the scheme coincides with the Godunov- and Engquist-Osher-type schemes.

\begin{remark}[Comparison to other general approaches in the literature]\label{rem:comparison}
As aforementioned typically schemes tailored to each application indiviually are studied.
In a recent work \cite{huang2023asymptotically}, the authors consider a more general approach for a class of numerical schemes.
Nevertheless, they restrict themselves to a specific subcase of \eqref{eq:ND}, i.e. $g(\rho)=\rho$ and $v(\rho)=1-\rho$.
The goal in \cite{huang2023asymptotically} is to derive asymptotically compatible numerical schemes and hence they assume further restrictions on the initial data (strictly positive without negative jumps) and an upper bound on the nonlocal range $\ndt$.
Besides considering a more general flux function, our approach to derive a numerical scheme also differs since we first approximate the nonlocal term, leaving us with only one nonlocal term in the flux, while in \cite{huang2023asymptotically} there may be an additional nonlocal term in the numerical flux function.
Hence, our scheme might be a special case of their scheme.
Nevertheless, some conditions for convergence are different, e.g. we assume Lipschitz type conditions on the numerical flux function in contrast to conditions on the partial derivatives.
Some conditions are also similar, such as the consistency or the first order derivatives with respect to the local variables.
For more details on the numerical scheme we refer to \cite[Assumption 4]{huang2023asymptotically}.
\end{remark}

\begin{remark}[Higher order extensions]
The scheme \eqref{eq:scheme} together with a numerical flux function satisfying Definition \ref{def:numscheme} is at most a first order approximation of the modeling equation \eqref{eq:ND}, see Section \ref{sec:numerics} for numerical convergence rates.
Nevertheless, using a semi-discrete approach and a more accurate approximation of the space, e.g. by a CWENO reconstruction procedure as in \cite{friedrich2019maximum},  together with an appropriate time stepping scheme and an accurate approximation of the nonlocal term, allows a higher order approximation of the solution.
One can simply follow the same approach as in \cite{friedrich2019maximum} with a numerical flux function that satisfies Definition \ref{def:numscheme}.
\end{remark}

\section{Main result}\label{sec:result}
In this section we prove that the assumptions on the numerical fluxes in Definition \ref{def:numscheme} together with the appropriate approximation of the nonlocal term in Definition \ref{def:NLtermdiscrete} and an appropriate Courant-Friedrichs-Levy (CFL) condition are sufficient to guarantee the convergence of the numerical scheme to the physically correct entropy solution.
Thereby, we consider the following CFL condition:
\begin{equation}\label{eq:CFL}
\lambda:=\frac{\Dt}{\Dx}\leq \frac{1}{\norm{G} \norm{v'} \gamma_0+L_1+L_2}.
\end{equation}
Now we are ready to present our main result:
\begin{theorem}\label{thm:mainresult}
Let the Assumptions \ref{ass:model} hold.  Then a numerical scheme \eqref{eq:scheme} with a numerical flux function fulfilling Definition \ref{def:numscheme} and the nonlocal term being approximated as in Definition \ref{def:NLtermdiscrete} converges under the CFL condition \eqref{eq:CFL} for $(\Dt,\Dx)\to 0$ to the unique weak entropy solution of \eqref{eq:ND} in the sense of Definition \ref{def:weakentropy}.\\
Moreover the sequence of approximate solutions fulfills the maximum principle for a given initial datum $\rho_j^0,\ j\in \Z$
\[\rv{\inf}_{j\in\Z}\, \rho_j^0\leq \rho_j^n\leq \rv{\sup}_{j\in\Z}\, \rho_j^0,\ j\in\Z,\ n\in\N.\]
\end{theorem}
The proof consists of proving several properties fulfilled by all schemes \eqref{eq:scheme} satisfying the Definitions \ref{def:NLtermdiscrete} and \ref{def:numscheme}.

\subsection{Maximum Principle}
We start by proving a discrete maximum principle.
Due to the involved nonlocal term, the numerical fluxes of Definition \ref{def:numscheme} are not monotone in every argument, in particular for $\rho^n_{j+k}$, $k=2,\dots,\Ne-1$, see also \cite{BlandinGoatin2016,huang2023asymptotically}.
Hence, we will have to rely on a different approach to show the maximum principle, similar to \cite{friedrich2020onetoone,friedrich2021network,friedrich2018godunov}. 
\begin{proposition}\label{prop:Maximumprinciple}
Under the Assumptions \ref{ass:model} and the CFL condition \eqref{eq:CFL}, for a given initial datum $\rho_j^0,\ j\in \Z$ with $\rho_m=\rv{\inf}_{j\in\Z}\, \rho_j^0$ and $\rho_M=\rv{\sup}_{j\in\Z}\, \rho_j^0$, the scheme \eqref{eq:scheme} satisfying the Definitions \ref{def:NLtermdiscrete} and \ref{def:numscheme} fulfills
\[\rho_m\leq \rho_j^n\leq \rho_M,\ j\in\Z,\ n\in\N.\]
\end{proposition}
\begin{proof}
We prove the claim by induction. For $n=0$ the claim is obvious, so we suppose
\[\rho_m\leq \rho_j^n\leq \rho_M,\ j\in\Z\]
holds for a fixed $n\in\N$.

Next, we consider
\begin{align}
    V_{\jmo}^{n}-V_{j}^{n}&= 
    v\left(\sum_{k=0}^{\Ne-1} \gamma_k \rho^n_{j+k}\right)-v\left(\sum_{k=0}^{\Ne-1} \gamma_k \rho^n_{j+k+1}\right)\nonumber\\
    &=v'(\xi_j)\left(\sum_{k=1}^{\Ne-1}(\gamma_k-\gamma_{k-1})\rho_\jpk^n-\gamma_{{\Ne-1}} \rho_{j+\Ne}^n+\gamma_0 \rho_j^n\right),\nonumber\\
    \intertext{\rv{where $\xi_j$ comes from the mean value theorem. Since} $\gamma_k\leq \gamma_{k-1}$ holds, we get}
    \rv{V_{\jmo}^{n}-V_{j}^{n}}&\leq v'(\xi_j)\left(\sum_{k=1}^{\Ne-1}(\gamma_k-\gamma_{k-1})\rho_M-\gamma_{{\Ne-1}}\rho_{M}+\gamma_0 \rho_j^n\right)\nonumber\\
    &\leq \norm{v'} \gamma_0 (\rho_M-\rho_j^n)\label{eq:boundV}.
\end{align}
Then, by the assumptions (iii) and (iv) on the numerical flux in Definition \ref{def:numscheme} and \eqref{eq:boundV} we have
\begin{align*}
F^n_\jmh-F^n_\jph\leq& G(\rho_M,\rho_j^n)V^n_{j-1}-G(\rho_j^n,\rho_M)V^n_j\\
=&  G(\rho_M,\rho_j^n)(V^n_{j-1}-V^n_j)-\left(G(\rho_j^n,\rho_M)-G(\rho_M,\rho_j^n)\right)V_j^n\\
\leq& \norm{G} \norm{v'} \gamma_0 (\rho_M-\rho_j^n) -F^n_\jph(\rho_j^n,\rho_M)\pm F^n_\jph(\rho_M,\rho_M)+F^n_\jph(\rho_M,\rho_j^n)\\
\leq& (\norm{G} \norm{v'} \gamma_0+L_1+L_2) (\rho_M-\rho_j^n).
\end{align*}
Hence, under the CFL condition \eqref{eq:CFL} we obtain
\[\rho_j^{n+1}\leq \rho_M.\]
Analogously to \eqref{eq:boundV}, we obtain
\begin{align*}
    V_{\jmo}^{n}-V_{j}^{n}&\geq \norm{v'} \gamma_0 (\rho_m-\rho_j^n)
\end{align*}
and 
\begin{align*}
F^n_\jmh-F^n_\jph\geq & (\norm{G} \norm{v'} \gamma_0+L_1+L_2) (\rho_m-\rho_j^n)
\end{align*}
to obtain under the CFL condition \eqref{eq:CFL} 
\[\rho_j^{n+1}\geq \rho_m.\]
\end{proof}
An immediate consequence of the maximum principle and in particular the positivity of the approximate solutions is that the discrete $L^1$-norm is conserved.
\begin{corollary}\label{corollary:L1}
    Under the Assumptions \ref{ass:model} and the CFL condition \eqref{eq:CFL} the approximate solutions computed by the numerical scheme \eqref{eq:scheme} satisfying the Definitions \ref{def:NLtermdiscrete} and \ref{def:numscheme} preserve the $L^1$ norm, i.e.
    \[\Delta x \sum_{j\in\Z}\abs{\rho_j^n}=\norm{\rho_0}_{L^1(\R)}.\]
\end{corollary}
\begin{remark}
    Proposition \ref{prop:Maximumprinciple} together with an approximation of the nonlocal term as in Definition \ref{def:NLtermdiscrete} guarantees the following bounds on the nonlocal term:
\[0\leq v(\rho_M)\leq V_j^n\leq v(\rho_m)\leq v(0).\]
\end{remark}

\subsection{Bounded variation estimates}
Next, we want to derive bounded variation (BV) estimates in space and time on the approximate solutions.
Therefore, we recall the definition of the piecewise constant function
\begin{align}\label{eq:rhodelta}
    \dr(t,x):=\sum_{n\in\N} \sum_{j\in\Z}\rho_j^n\,\caratt{[t^n,\,t^{n+1})\times[x_\jmh,\,x_\jph)}(t,x).
\end{align}
As for the analytical solution obtained in \cite{chiarello2018global} we only recover estimates which increase exponentially in time.
\begin{proposition}[BV estimate in space] Let the Assumptions \ref{ass:model} and the CFL condition \eqref{eq:CFL} hold. Let $\dr$ be defined as in \eqref{eq:rhodelta} with a scheme \eqref{eq:scheme} satisfying the Definitions \ref{def:NLtermdiscrete} and \ref{def:numscheme}, then we obtain for any finite time $t^n>0$ the following bound on the total variation in space
\begin{align*}
\rv{\TV(\rho^{\Delta}(t^n,\cdot))=\sum_{j\in\Z}\abs{\rho_\jpo^n-\rho_j^n}}\leq \exp\left(t^n \wt(0) \left(\norm{v'}\left(2\norm{g}+\norm{g'}\,\norm{\rho}\right)+2 \norm{v''}\,\norm{g}\,\norm{\rho}\right)\right) \TV(\rho_0).
\end{align*}
\end{proposition}
\begin{proof}
We can rewrite the scheme as
\begin{align*}
\rho_j^{n+1}=&\rho_j^n+\lambda \left(G(\rho_\jmo^n,\rho_j^n) V_\jmo^n-G(\rho_j^n,\rho_\jpo^n) V_j^n\right)\\
=&\rho_j^n+a_\jmh^n (\rho_\jmo^n-\rho_j^n)+b_\jph^n(\rho_\jpo^n-\rho_j^n)+\lambda g(\rho_j^n) \left(V_\jmo^n-V_j^n\right)
\end{align*}
with 
\begin{align}
a_\jmh^n&=\begin{cases}
\lambda V_\jmo^n\frac{G(\rho_\jmo^n,\rho_j^n)-G(\rho_j^n,\rho_j^n)}{\rho_\jmo^n-\rho_j^n},\quad &\text{if }\rho_\jmo^n\neq \rho_j^n,\\
0,\quad &\text{if }\rho_\jmo^n= \rho_j^n,\\
\end{cases}\\
b_\jph^n&=\begin{cases}
\lambda V_j^n\frac{G(\rho_j^n,\rho_\jpo^n)-G(\rho_j^n,\rho_j^n)}{\rho_j^n-\rho_\jpo^n},\qquad &\text{if }\rho_\jpo^n\neq \rho_j^n,\\
0,\quad &\text{if }\rho_\jpo^n= \rho_j^n,\\
\end{cases}
\end{align}
with (due to the montonicity and weak Lipschitz continuity) $0\leq a_\jmh^n\leq L_1\lambda$ and $0\leq a_\jmh^n\leq L_2\lambda$.
Next we define
\[\Delta_j^n=\rho_\jpo^n-\rho_j^n.\]
Then, we obtain
\begin{align*}
\Delta_j^{n+1}=&(1-a_\jph^n-b_\jph^n)\Delta_j^n+b_{j+\frac{3}{2}}^n\Delta_\jpo^n+a_\jmh^n \Delta_\jmo^n\\
&+\lambda\left(g(\rho_\jpo^n)(V_j^n-V_\jpo^n)-g(\rho_j^n)(V_\jmo^n-V_j^n)\right).
\end{align*}
Adding a zero the last term can be written as
\begin{align*}
g(\rho_\jpo^n)(V_j^n-V_\jpo^n)-g(\rho_j^n)(V_\jmo^n-V_j^n)=&g'(\zeta_j)\Delta_j^n (V_j^n-V_\jmo^n)\\
&+g(\rho_j^n)\left(-V_\jpo^n+2V_j^n-V_\jmo^n\right),
\end{align*}
with $\zeta_j$ coming form the mean value theorem such that $g'(\zeta_j)=g(\rho_{j+1}^n)-g(\rho_j^n)$.
Now, we consider the difference in the velocities (with $\xi_j$ also coming from the mean value theorem) and
obtain
\begin{align*}
    V_j^n-V_{j-1}^n=v'(\xi_j)\sum_{k=0}^{\Ne-1} \gamma_k \Delta_\jpk^n\quad\text{and}\quad V_\jpo^n-V_{j}^n=v'(\xi_\jpo)\sum_{k=0}^{\Ne-1} \gamma_k \Delta_{\jpk+1}^n.
\end{align*}
This leads us to
\begin{align*}
-V_\jpo^n+2V_j^n-V_\jmo^n=&\gamma_0 v'(\xi_j)\Delta_j^n-\gamma_{{\Ne-1}}v'(\xi_\jpo)\Delta_{j+\Ne}^n+\sum_{k=1}^{\Ne-1} (\gamma_k v'(\xi_j)-\gamma_{k-1} v'(\xi_\jpo))\Delta_{j+k}^n\\
=&\gamma_0 v'(\xi_j)\Delta_j^n-\gamma_{{\Ne-1}}v'(\xi_\jpo)\Delta_{j+\Ne}^n
+\sum_{k=1}^{\Ne-1} (\gamma_k-\gamma_{k-1}) v'(\xi_j)\Delta_{j+k}^n\\
&+\sum_{k=1}^{\Ne-1} \gamma_{k-1}(v'(\xi_j)- v'(\xi_\jpo))\Delta_{j+k}^n.
\end{align*}
As several terms are positive due to the CFL condition \eqref{eq:CFL}, we can put everything together by using also \eqref{eq:boundV} to obtain
\begin{align*}
\abs{\Delta_j^{n+1}}\leq&(1-a_\jph^n-b_\jph^n+\lambda\gamma_0\norm{v'}\left(\norm{g}+\norm{g'}\,\norm{\rho}\right))\abs{\Delta_j^n}+b_{j+\frac{3}{2}}^n\abs{\Delta_\jpo^n}+a_\jmh^n \abs{\Delta_\jmo^n}\\ 
&+\gamma{_{\Ne-1}}\norm{v'}\,\norm{g}\lambda \abs{\Delta_{j+\Ne}^n}+\sum_{k=0}^{\Ne-1} (\gamma_{k-1}-\gamma_k)\norm{v'}\,\norm{g}\lambda\abs{\Delta_{j+k}^n}\\
&+\norm{g}\lambda\sum_{k=1}^{\Ne-1} \gamma_{k-1}\abs{v'(\xi_j)- v'(\xi_\jpo)}\abs{\Delta_{j+k}^n}.
\end{align*}
Summing over $j\in \Z$, rearranging the indices and using $\lambda\gamma_0\leq \Dt \wt(0)$ leaves us with
\begin{small}
\begin{align*}
\sum_{j\in\Z}\abs{\Delta_j^{n+1}}\leq \sum_{j\in\Z}\left( 1+\Delta t \wt(0) \norm{v'}\left(2\norm{g}+\norm{g'}\norm{\rho}\right)+\lambda \norm{v''}\norm{g}\sum_{k=1}^{\Ne-1} \gamma_{k-1}\abs{\xi_{j-k}- \xi_{\jpo-k}}\right)\abs{\Delta_j^n}.
\end{align*}
\end{small}
We need to estimate $\abs{\xi_j-\xi_\jpo}$.
As we used the mean value theorem we know that $\xi_j$ is between
\[R_j^n:=\sum_{k=0}^{\Ne-1} \gamma_k \rho^n_{j+k}\quad \text{and}\quad R_\jpo^n.\]
Hence, we obtain
\begin{align}
\abs{\xi_j-\xi_\jpo}\leq \abs{\xi_j-R_\jpo^n}+\abs{R_\jpo^n-\xi_\jpo}\leq\abs{R_\jpo^n-R_j^n}+\abs{R_{j+2}^n-R_\jpo^n}\leq 2\gamma_0 \norm{\rho}.\label{eq:estimate1}
\end{align}
This follows by similar calculations as done to obtain \eqref{eq:boundV}.

Finally, this leaves us with
\begin{align*}
\sum_{j\in\Z}\abs{\Delta_j^{n+1}}\leq \left( 1+\Delta t \wt(0) \left(\norm{v'}\left(2\norm{g}+\norm{g'}\,\norm{\rho}\right)+2 \norm{v''}\,\norm{g}\,\norm{\rho}\right)\right) \sum_{j\in\Z}\abs{\Delta_j^n}.
\end{align*}
Therefore, we can conclude 
\begin{align*}
\TV(\dr(t^{n+1},\cdot))=\sum_{j\in\Z}\abs{\Delta_j^{n+1}}\leq & \left( 1+\Delta t \wt(0) \left(\norm{v'}\left(2\norm{g}+\norm{g'}\,\norm{\rho}\right)+2 \norm{v''}\,\norm{g}\,\norm{\rho}\right)\right)^{t^{n+1}/\Delta t} \sum_{j\in\Z}\abs{\Delta_j^0}\\
\leq & \exp \left( t^{n+1} \wt(0) \left(\norm{v'}\left(2\norm{g}\norm{g'}\,\norm{\rho}\right)+2 \norm{v''}\,\norm{g}\,\norm{\rho}\right)\right) \sum_{j\in\Z}\abs{\Delta_j^0}\\
\leq & \exp \left( t^{n+1} \wt(0) \left(\norm{v'}\left(2\norm{g}+\norm{g'}\,\norm{\rho}\right)+2 \norm{v''}\,\norm{g}\,\norm{\rho}\right)\right) \TV(\rho_0),
\end{align*}
where we used \cite[Remark 5.4]{eymard2000finitevolume} for the last inequality.
\end{proof}
\begin{remark}
Due to the presence of $g(\rho)$, we cannot rewrite the scheme directly into the form proposed in \cite[Theorem A.1]{abreu2022lagrangian}.
Nevertheless, our proof is similar to the one there and also to the one proposed in \cite{eymard2000finitevolume} for the local case.
\end{remark}
\begin{remark}
Note that the obtained estimate on the total variation is very similar to the one in \cite[Proposition 3.2]{chiarello2018global}.
In fact, the main difference comes from the estimate \eqref{eq:estimate1}, which gives the factor $2 \norm{v''}\,\norm{g}\,\norm{\rho}$.
In contrast, the factor $7/2 \norm{v''}\,\norm{g}$ is obtained in \cite{chiarello2018global}.
Hence, for $\norm{\rho}=\norm{\rho_0}\leq 7/4$, which is the case for most applications since $\norm{\rho_0}\leq 1$ normally holds, our estimate provides a slightly smaller upper bound on the analytical solution.
\end{remark}
Now, we are able to provide an estimate on the continuity in time.
\begin{proposition}[Time continuity estimates]
 Let the Assumptions \ref{ass:model} and the CFL condition \eqref{eq:CFL} hold. Let $\dr$ be defined as in \eqref{eq:rhodelta} with a scheme \eqref{eq:scheme} satisfying the Definitions \ref{def:NLtermdiscrete} and \ref{def:numscheme}, then we obtain the following estimate
\begin{align*}
\Delta x \sum_{j\in\Z} \abs{\rho_j^{n+1}-\rho_j^n}\leq \Delta t \left(\norm{G}\,\norm{v'}+L_1+L_2\right)\sum_{j\in\Z} \abs{\rho_\jpo^n-\rho_j^n}
\end{align*}
\end{proposition}
\begin{proof}
We observe
\begin{align*}
\rho_j^{n+1}-\rho_j^n=&\lambda \left(G(\rho_\jmo^n,\rho_j^n)V_\jmo^n-G(\rho_j^n,\rho_\jpo^n)V_j^n\right)\\
=&\lambda\left(G(\rho_\jmo^n,\rho_j^n)(V_\jmo^n-V_j^n)+(G(\rho_\jmo^n,\rho_j^n)\pm G(\rho_j^n,\rho_j^n)-G(\rho_j^n,\rho_\jpo^n))V_j^n\right).
\end{align*}
Hence, we obtain
\begin{align*}
\abs{\rho_j^{n+1}-\rho_j^n}\leq \lambda \left(\norm{G}\norm{v'}\sum_{k=0}^{\Ne-1}\gamma_k \abs{\rho_{j+k+1}^n-\rho_{j+k}^n}+L_2\abs{\rho_{j+1}^n-\rho_j^n}+L_1\abs{\rho_j^n-\rho_\jmo^n}\right).
\end{align*}
Summing over $j\in\Z$ and rearranging the indices yields then the claim.
\end{proof}
From the two latter propositions we can obtain an estimate on the bounded variation in space and time:
\begin{corollary}\label{cor:BVspacetime}
 Let the Assumptions \ref{ass:model} and the CFL condition \eqref{eq:CFL} hold. Let $\dr$ be defined as in \eqref{eq:rhodelta} with a scheme \eqref{eq:scheme} satisfying the Definitions \ref{def:NLtermdiscrete} and \ref{def:numscheme}, then we obtain for any $n\in\N$ 
\begin{align*}
    &\sum_{m=0}^{n-1}\sum_{j\in\Z} \Delta t \abs{\rho_{j+1}^n-\rho_j^n}+\Delta x \abs{\rho_j^{n+1}-\rho_j^n}\leq\\
    &n\Delta t  \left(1+\norm{G}\,\norm{v'}+L_1+L_2\right)\exp \left( t^{n} \wt(0) \left(\norm{v'}\left(2\norm{g}+\norm{g'}\,\norm{\rho}\right)+2 \norm{v''}\,\norm{g}\,\norm{\rho}\right)\right) \TV(\rho_0).
\end{align*}
\end{corollary}

\subsection{Discrete entropy inequality and convergence}
To prove the convergence against the entropy solution we need an equivalent discrete formulation.
For a constant $k\in\R$ we denote
\[\mathcal{F}^k_{j+1/2}(u,w)=F_\jph^n(u\wedge k,w\wedge k)-F_\jph^n(u\vee k,w\vee k)\]
with $a\wedge b=\max(a,b)$ and $a \vee b=\min(a,b)$.
\begin{proposition}\label{prop:discreteentropy}
 A numerical scheme \eqref{eq:scheme} satisfying the Definitions \ref{def:NLtermdiscrete} and \ref{def:numscheme}, fulfills
$$\lvert\rho_{j}^{n+1}-k\rvert-|\rho_j^n-k|+\lambda(\mathcal{F}^k_{j+1/2}(\rho_j^n,\rho_{j+1}^n)-\mathcal{F}^k_{j-1/2}(\rho_{j-1}^n,\rho_j^n)) + \frac{\lambda}{2} \sign({\rho_{j}^{n+1}-k}) g(k) (V^{n}_{j}-V^{n}_{j-1})\leq 0.$$
\end{proposition}
\begin{proof}
We define 
\[H(u,w,z)(u,w,z)=w-\lambda(F^n_{j+\frac{1}{2}}(w,z)-F^n_{j-\frac{1}{2}}(u,w)),\]
which is monotone in each variable due to our assumptions on the numerical flux in Definition \ref{def:numscheme}.
Now, we can follow \cite{BlandinGoatin2016,chiarello2018global,friedrich2018godunov} to obtain the desired claim.
We omit the details.
\end{proof}
Finally, we are ready to state the proof of convergence for the considered class of numerical schemes.
\begin{proof}[Proof of Theorem \ref{thm:mainresult}]
Due to Proposition \ref{prop:Maximumprinciple} and Corollary \ref{cor:BVspacetime}, we can apply Helly's theorem as presented in \cite[Lemma 5.6]{eymard2000finitevolume} and conclude the existence of a subsequence of approximate solutions that converge to some $\rho\in (L^\infty \cap \BV)(\R^+\times \R; \R)$.
In addition, thanks to Corollary \ref{corollary:L1}, we have $\rho \in L^1(\R^+\times \R;\R)$.
Using the discrete entropy inequality from Proposition \ref{prop:discreteentropy} and Lax-Wendroff type arguments similar to \cite{BlandinGoatin2016} shows that the limit function $\rho$ is a weak entropy solution of \eqref{eq:ND} in the sense of Definition \ref{def:weakentropy}.
The discrete maximum principle is given directly in Proposition \ref{prop:Maximumprinciple} 
\end{proof}

\section{Numerical examples}\label{sec:numerics}
In this section we complement the theoretical results with numerical examples and consider the numerical convergence rates of several numerical schemes satisfying Definition \ref{def:numscheme}.
In Section \ref{sec:numtraffic} we focus on a traffic flow model with Arrhenius look-ahead dynamics belonging to \eqref{eq:ND}.
Section \ref{sec:numsedimentation} considers our approach for a nonlocal sedimentation model which has a different convolution than \eqref{eq:ND:conv}.

\subsection{Arrhenius-type look-ahead traffic flow model}\label{sec:numtraffic}
We first consider a numerical example to compare the schemes derived above, namely the Lax-Friedrichs-type scheme \eqref{eq:LxFmod}, the Godunov-type scheme \eqref{eq:Godunov}, and the Engquist-Osher-type scheme \eqref{eq:EOflux}. 
We also compare the numerical schemes with the Lax-Friedrichs-type scheme of \cite{chiarello2018global}, which is slightly different from our numerical scheme \eqref{eq:LxFmod}.
We choose the diffusion parameter $\alpha$ of \cite{chiarello2018global} for a fair numerical comparison between the two Lax-Friedrichs-type schemes.
Furthermore, the CFL condition is given as in \cite{chiarello2018global}, since this is the most restrictive one.
We will use the following example:
\[g(\rho)=\rho(1-\rho),\quad v(\rho)=\exp(-\rho),\quad \wt(x)=
2(\ndt-x)/\ndt^2,\quad \ndt=0.1.\]
Hence, we consider the Arrhenius-type look-ahead model with a linear decreasing function. 
In particular, in contrast to the case of the nonlocal LWR model, i.e. $g(\rho)=\rho$, all the schemes \eqref{eq:LxFmod}--\eqref{eq:EOflux} differ here.
We consider the discontinuous initial condition:
\[\rho_0(x)=\begin{cases}0.8\quad & x\in[3/4,5/4],\\
0\quad &\text{else}
\end{cases}
\]
at the final time $T=0.5$.
In Figure \ref{fig:trafficconv} we see the approximate solutions on the left and a zoom into the spatial domain on the right, together with a reference solution computed by the Godnuov-type scheme and $\Dx=0.01\cdot2^{-6}$.
It can be seen that the Engquist-Osher and Godunov-type schemes are already closer to the reference solution, while the Lax-Friedrichs-type schemes have a higher numerical diffusion.
However, the newly derived Lax-Friedrichs-scheme \eqref{eq:LxFmod} is closer to the reference solution than the Lax-Friedrichs scheme of \cite{chiarello2018global}.
In addition, the zoom on the right side shows that the solutions of the two Lax-Friedrichs-type schemes, but also of the Engquist-Osher and Godunov-type schemes, differ.

\begin{figure}[h]
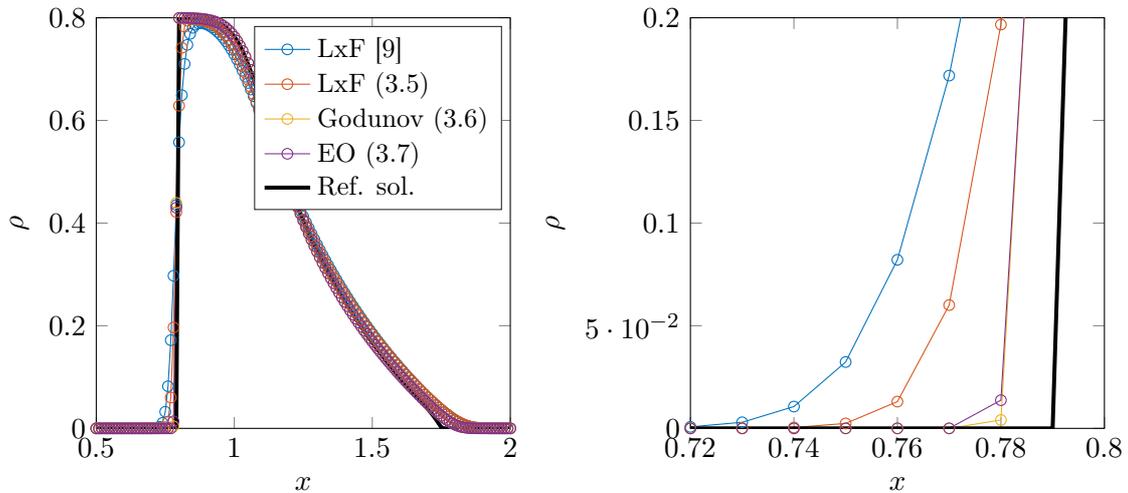

\setlength{\fwidth}{0.7\textwidth}
\begin{center}
\input{ArrheniusConvergence}
\input{ArrheniusConvergenceZoom}
\end{center}
\caption{Approximate solutions obtained by the different schemes from \cite{chiarello2018global} and \eqref{eq:LxFmod}--\eqref{eq:EOflux} with $\Dx=0.01$ and a reference solution (ref. sol.), left, and a zoom into the spatial domain, right.}\label{fig:trafficconv}
\end{figure}

Furthermore, we test the numerical convergence rate in the $L^1$-norm in comparison to the reference solution, still computed by the Godnuov type scheme and $\Dx=0.01\cdot2^{-6}$.
We choose $\Dx=0.01\cdot2^{-n}$ for $n=0,\dots,5$.
Theoretically, we expect a convergence rate of one \rv{and
Table \ref{tab:Arrhenius} shows the corresponding errors. The numerical convergence rates suggest the expected order of one.}
We note that the errors of the Godunov-type scheme \eqref{eq:Godunov} are the smalles.

\begin{table}
\centering
\begin{small}
\begin{tabular}{c| c c c c | c c c c}
&\multicolumn{4}{c|}{$L^1$-errors}&\multicolumn{4}{c}{convergence rate}\\
\hline
$n$& LxF \cite{chiarello2018global} &	LxF \eqref{eq:LxFmod} &	Godunov \eqref{eq:Godunov}&EO \eqref{eq:EOflux}& LxF \cite{chiarello2018global} &	LxF \eqref{eq:LxFmod} &	Godunov \eqref{eq:Godunov}&EO \eqref{eq:EOflux}\\
\hline\hline
0&0.0343&	0.0248&	0.0085&	0.0085& & & &\\
1&0.0178& 	0.0127&	0.0026&	0.0033&0.9482&    0.9739&    	1.6770&1.3738\\
2&0.0091&	0.0064&	0.0013&	0.0016&0.9644&		0.9792&    	1.0711&1.0578\\
3&0.0049&	0.0035&6.6881e-04&8.2489e-04&0.8882&    0.8619&    	0.9133&0.9280\\
4&0.0026&	0.0019&3.4622e-04&4.2017e-04&0.9196&		0.9004&		0.9499&0.9730\\
5&0.0014&	0.0010&1.8495e-04&2.1174e-04&0.9203&		0.8974&		0.9045&0.9887\\
\hline
\end{tabular}
\end{small}
\caption{$L^1$-errors of the different schemes in comparison to a reference solution and the numerical convergence rates for $\Dx=0.01\cdot2^{-n}$.}\label{tab:Arrhenius}
\end{table}

\subsection{Further nonlocal problems: A sedimentation model}\label{sec:numsedimentation}
In this part we want to \rv{ show that the proposed strategy can be applied to other nonlocal problems and that it works well. However, we note that we will not provide a rigorous convergence proof as before.}
We will focus on a sedimentation model.

But first, let us mention that the strategy was already successfully used to derive a Godunov-type scheme in  \cite{friedrich2018godunov} for another traffic flow model including a \rv{rigorous} convergence proof.
In addition, in \cite{colombo2019role} a Godunov scheme for a nonlocal Burgers equation is derived and numerically evaluated in a similar manner.
In both these models the resulting Engquist-Osher scheme is the same as the Godunov-type scheme.
Furthermore, in \cite{burger2022hilliges} a Hilliges-Weidlich-type scheme is derived for a specific nonlocal problem appearing for example in sedimentation.
Due to the modeling equations, the nonlocal term is different from \eqref{eq:ND:conv}, but after approximating the nonlocal term, the numerical flux function of \cite{burger2022hilliges} satisfies Definition \ref{def:numscheme}.

Now, we want to apply the approach to the nonlocal sedimentation model presented in \cite{betancourt2011nonlocal}.
In this case we have 
\[g(\rho)=\rho(1-\rho),\quad v(\rho)=(1-\rho)^4\]
but the nonlocal term is
\[\rv{(\rho\ast\wt)(x)}=\int_{x-2\ndt}^{x+2\ndt}\rho(t,y)\wt(y-x)dx,\]
with a symmetric kernel.
Hence, it does not belong to the class of \eqref{eq:ND}.
Nevertheless, the well-posedness can be proved. For further details we refer to \cite{betancourt2011nonlocal}.
To construct numerical flux functions we apply our approach and approximate the nonlocal term by
\begin{align*}
    V_j^n= v\left(\sum_{k=-\Ne}^{\Ne-1} \gamma_k \rho_{j+k+1}^n\right)
\end{align*}
with $\Ne=2\ndt/\Delta x$ and the weights $\gamma_k$ are defined as before.
As proved in \cite{betancourt2011nonlocal}, the solution inherits a maximum principle between $[0,1]$.
Hence, as long as it is kept by the numerical schemes, the velocity of the nonlocal term remains positive.
So we can apply the numerical flux functions \eqref{eq:LxFmod}--\eqref{eq:EOflux}.
We repeat the numerical test of \cite[Example 2]{betancourt2011nonlocal} for the initial data:
\[\rho_0(x)=\begin{cases}0.01\quad & x>0.2,\\
0\quad &\text{else}
\end{cases}\quad \text{and}\quad \rho_0(x)=\begin{cases}0.6\quad & x>0.2,\\
0\quad &\text{else}\end{cases}.\]
The kernel is set to
\begin{align*}
\wt(x)=\frac{K(x/\ndt)}{\ndt}\quad K(x)=\frac{3}{8}\left(1-\frac{x^2}{4}\right)\text{for }|x|<2\text{ and }K(x)=0\text{ else}.
\end{align*}
We compare our numerical schemes with the Lax-Friedrichs scheme used in \cite{betancourt2011nonlocal}.
For the latter one the convergence is proven.
\rv{As in \cite{betancourt2011nonlocal} we choose $\lambda=0.2$.
Table \ref{tab:Sedi} shows the error terms at $T=3$ for the first initial condition.
The reference solution is computed using the Lax-Friedrichs scheme of \cite{betancourt2011nonlocal} with $\Dx=0.2\cdot 2^{-9}$.
The Engquist-Osher and Godunov type schemes obtain their expected order of convergence.
We note that they coincide in this example.
The numerical convergence rates for the Lax-Friedrichs type schemes are lower.
Nevertheless, the error terms of all schemes suggest the convergence to the correct entropy solution.
The solution with $\Dx=0.2\cdot 2^{-5}$ can be seen in Figure \ref{fig:sedimentation}, left.
The difference between the two Lax-Friedrichs type schemes is not visible.
\begin{table}[tbh]
\centering
\begin{small}
\begin{tabular}{c| c c c c | c c c c}
&\multicolumn{4}{c|}{$L^1$-errors}&\multicolumn{4}{c}{convergence rate}\\
\hline
$n$& LxF \cite{betancourt2011nonlocal} &	LxF \eqref{eq:LxFmod} &	Godunov \eqref{eq:Godunov}&EO \eqref{eq:EOflux}& LxF \cite{betancourt2011nonlocal} &	LxF \eqref{eq:LxFmod} &	Godunov \eqref{eq:Godunov}&EO \eqref{eq:EOflux}\\
\hline\hline
0     & 0.0134 & 0.0133 &     0.0055 &     0.0055 & & & &\\
1     & 0.0090 & 0.0089 &     0.0033 &     0.0033 &   0.5634 &   0.5715 &   0.7239 &   0.7239\\
2     & 0.0063 & 0.0062 &     0.0022 &     0.0022 &   0.5248 &   0.5210 &   0.5588 &   0.5588\\
3     & 0.0043 & 0.0043 &     0.0015 &     0.0015 &   0.5455 &   0.5469 &   0.5702 &   0.5702\\
4     & 0.0029 & 0.0029 & 9.3067e-04 & 9.3067e-04 &   0.5725 &   0.5740 &   0.7010 &   0.7010\\
5     & 0.0019 & 0.0019 & 4.5755e-04 & 4.5755e-04 &   0.5885 &   0.5864 &   1.0243 &   1.0243\\
\hline
\end{tabular}
\end{small}
\caption{$L^1$-errors of the different schemes in comparison to a reference solution and the numerical convergence rates for $\Dx=0.2\cdot2^{-n}$ at $T=3$ for the first initial condition.}\label{tab:Sedi}
\end{table}
\begin{figure}[tbh]
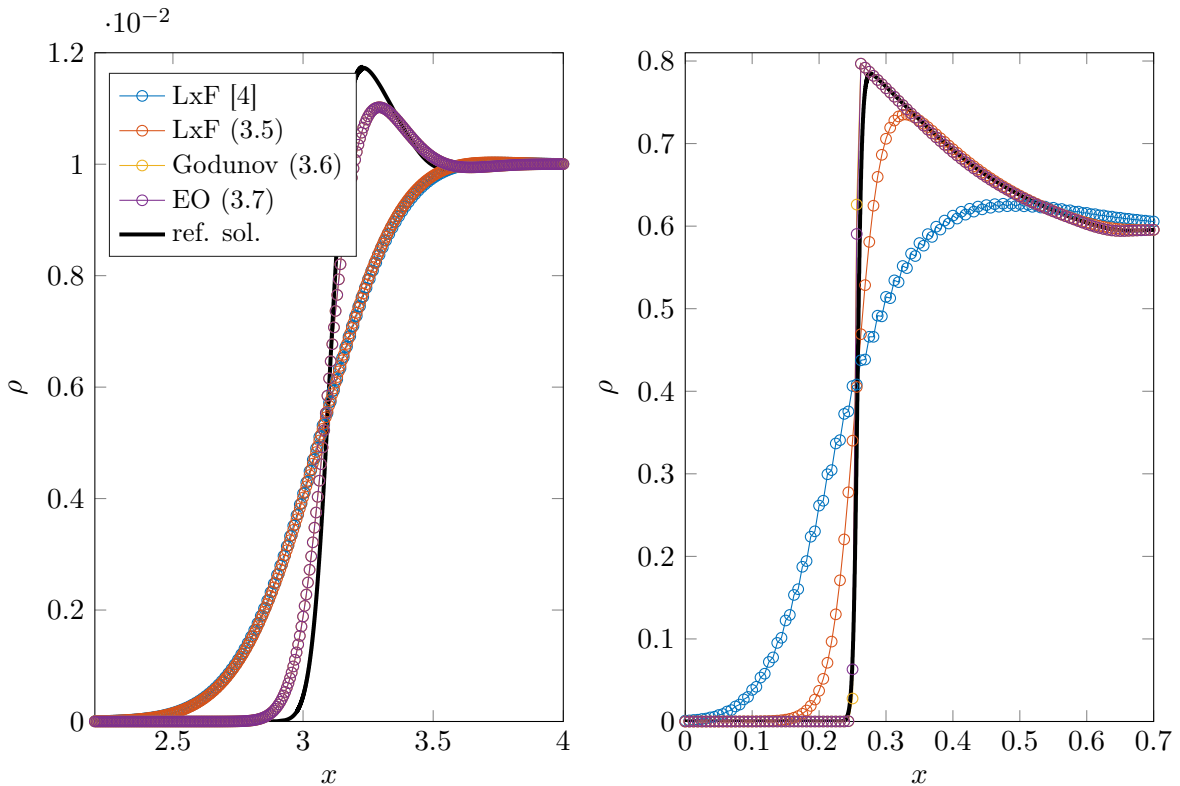

\setlength{\fwidth}{0.6\textwidth}
\begin{center}
\input{Sedimentation001}
\input{Sedimentation06}
\end{center}
\caption{\rv{Approximate solutions obtained by the different schemes from \cite{betancourt2011nonlocal} and \eqref{eq:LxFmod}--\eqref{eq:EOflux} with $\Dx=0.2\cdot 2^{-5}$ and a reference solution (ref. sol.) at $T=3$ for the first initial condition, left and at $T=1$ for the second initial condition, right.}}\label{fig:sedimentation}
\end{figure}
Table \ref{tab:Sedi2} presents the error terms for the second initial condition at $T=1$.
The reference solution is computed using the Lax-Friedrichs scheme of \cite{betancourt2011nonlocal} with $\Dx=0.2\cdot 2^{-10}$.
Again, the error terms indicate the convergence to the entropy solution.
In this example, we see that the newly derived Lax-Friedrichs scheme \eqref{eq:LxFmod} is closer to the reference solution than the scheme from \cite{betancourt2011nonlocal}.
Furthermore, the Engquist-Osher and Godunov type schemes differ. 
The numerical convergence rate of the newly derived Lax-Friedrichs scheme \eqref{eq:LxFmod} seems to converge to one.
The convergence rates of the Engquist-Osher and Godunov type schemes drop below one for $\Dx=0.2\cdot 2^{-5}$.
This behavior occurs because these schemes have less numerical viscosity than the Lax-Friedrichs scheme of \cite{betancourt2011nonlocal}.
Hence, they appear to provide a more accurate solution than the reference solution, even for a much larger step size.
This can be seen in Figure \ref{fig:sedimentation}, right.
It can also be seen that the newly derived Lax-Friedrichs scheme \eqref{eq:LxFmod} outperforms the Lax-Friedrichs scheme of \cite{betancourt2011nonlocal} in terms of accuracy.
\begin{table}[tbh]
\centering
\begin{small}
\begin{tabular}{c| c c c c | c c c c }
&\multicolumn{4}{c|}{$L^1$-errors}&\multicolumn{4}{c}{convergence rate}\\
\hline
$n$& LxF \cite{betancourt2011nonlocal} &	LxF \eqref{eq:LxFmod} &	Godunov \eqref{eq:Godunov}&EO \eqref{eq:EOflux}& LxF \cite{betancourt2011nonlocal} &	LxF \eqref{eq:LxFmod} &	Godunov \eqref{eq:Godunov}&EO \eqref{eq:EOflux}\\
\hline\hline
0    &0.4479 & 0.2005 &  0.1164 & 0.1163 & & & &\\
1    &0.3182 & 0.1429 &  0.0407 & 0.0423 &   0.4933 &   0.4891 &   1.5145 &   1.4597\\
2    &0.2118 & 0.0905 &  0.0333 & 0.0333 &   0.5876 &   0.6580 &   0.2915 &   0.3454\\
3    &0.1463 & 0.0583 &  0.0140 & 0.0140 &   0.5337 &   0.6344 &   1.2450 &   1.2453\\
4    &0.0997 & 0.0335 &  0.0060 & 0.0060 &   0.5534 &   0.8021 &   1.2192 &   1.2190\\
5    &0.0661 & 0.0168 &  0.0040 & 0.0036 &   0.5917 &   0.9973 &   0.5890 &   0.7582\\
\hline
\end{tabular}
\end{small}
\caption{\rv{$L^1$-errors of the different schemes in comparison to a reference solution and the numerical convergence rates for $\Dx=0.2\cdot2^{-n}$ at $T=1$ for the second initial condition.}}\label{tab:Sedi2}
\end{table}
}

\section{Conclusion}\label{sec:conclusion}
In this work we have presented a general approach to treat a specific class of nonlocal conservation laws.
The nonlocal term is approximated by an appropriate quadrature rule, and then explicit conditions for a numerical flux function on the reduced problem are presented that guarantee convergence.
These conditions share some similarities with the class of monotone schemes for local conservation laws.
Numerical examples show that the approach can be applied to other classes of nonlocal conservation laws.
Furthermore, other schemes from the literature use the same approach for different nonlocal conservation laws.

Future work could include generalizing the concept to other nonlocal terms as well as nonlocal multidimensional balance laws.
Another interesting topic would be the behavior of the schemes for $\ndt \to 0$. 
Depending on the modeling equation considered, a convergence to a local conservation law can be obtained on the analytic level.
Similar to the work done in \cite{huang2023asymptotically} it would be interesting to see, if the derived class of numerical schemes is asymptotically compatible, i.e. converges to the correct local solution for $\ndt\to 0$.

\section*{Acknowledgments}
J. F. is supported by the German Research Foundation (DFG) under grant HE 5386/18-1, 19-2, 22-1, 23-1. S. R. is supported by IIPE, Visakhapatnam, India, under the IRG grant number IIPE/DORD/IRG/001 and  NBHM, DAE, India (Ref. No. 02011/46/2021 NBHM(R.P.)/R \& D II/14874).

\section*{Conflict of interest}

The authors declare there is no conflict of interest.
\bibliographystyle{siam}
\bibliography{mysources.bib}

\begin{thebibliography}{10}

\bibitem{abreu2022lagrangian}
{\sc E.~Abreu, J.~Juajibioy, W.~Lambert, et~al.}, {\em Lagrangian-eulerian
  approach for nonlocal conservation laws}, Journal of Dynamics and
  Differential Equations,  (2022), pp.~1--47.

\bibitem{aggarwal2015nonlocal}
{\sc A.~Aggarwal, R.~M. Colombo, and P.~Goatin}, {\em Nonlocal systems of
  conservation laws in several space dimensions}, SIAM J. Numer. Anal., 53
  (2015), pp.~963--983.

\bibitem{amorim2015numerical}
{\sc P.~Amorim, R.~M. Colombo, and A.~Teixeira}, {\em On the numerical
  integration of scalar nonlocal conservation laws}, ESAIM Math. Model. Numer.
  Anal., 49 (2015), pp.~19--37.

\bibitem{betancourt2011nonlocal}
{\sc F.~Betancourt, R.~B\"urger, K.~H. Karlsen, and E.~M. Tory}, {\em On
  nonlocal conservation laws modelling sedimentation}, Nonlinearity, 24 (2011),
  pp.~855--885.

\bibitem{BlandinGoatin2016}
{\sc S.~Blandin and P.~Goatin}, {\em Well-posedness of a conservation law with
  non-local flux arising in traffic flow modeling}, Numer. Math., 132 (2016),
  pp.~217--241.

\bibitem{burger2022hilliges}
{\sc R.~Bürger, H.~Contreras, and L.~Villada}, {\em A hilliges-weidlich-type
  scheme for a one-dimensional scalar conservation law with nonlocal flux},
  Netw. Heterog. Media,  (to appear).

\bibitem{ChalonsGoatinVillada2018}
{\sc C.~Chalons, P.~Goatin, and L.~M. Villada}, {\em High-order numerical
  schemes for one-dimensional nonlocal conservation laws}, SIAM J. Sci.
  Comput., 40 (2018), pp.~A288--A305.

\bibitem{friedrich2020onetoone}
{\sc F.~A. Chiarello, J.~Friedrich, P.~Goatin, S.~G{\"o}ttlich, and O.~Kolb},
  {\em A non-local traffic flow model for 1-to-1 junctions}, European J. Appl.
  Math., 31 (2020), pp.~1029--1049.

\bibitem{chiarello2018global}
{\sc F.~A. Chiarello and P.~Goatin}, {\em Global entropy weak solutions for
  general non-local traffic flow models with anisotropic kernel}, ESAIM Math.
  Model. Numer. Anal., 52 (2018), pp.~163--180.

\bibitem{chiarello2019multiclass}
\leavevmode\vrule height 2pt depth -1.6pt width 23pt, {\em Non-local
  multi-class traffic flow models}, Netw. Heterog. Media, 14 (2019),
  pp.~371--387.

\bibitem{chiarello2020lagrangian}
{\sc F.~A. Chiarello, P.~Goatin, and L.~M. Villada}, {\em
  Lagrangian-antidiffusive remap schemes for non-local multi-class traffic flow
  models}, Computational and Applied Mathematics, 39 (2020), pp.~1--22.

\bibitem{colombo2019role}
{\sc M.~Colombo, G.~Crippa, M.~Graff, and L.~V. Spinolo}, {\em On the role of
  numerical viscosity in the study of the local limit of nonlocal conservation
  laws}, ESAIM: M2AN, 55 (2021), pp.~2705--2723.

\bibitem{colombo2012class}
{\sc R.~M. Colombo, M.~Garavello, and M.~L\'ecureux-Mercier}, {\em A class of
  nonlocal models for pedestrian traffic}, Math. Models Methods Appl. Sci., 22
  (2012), p.~1150023.

\bibitem{crandall1980monotone}
{\sc M.~G. Crandall and A.~Majda}, {\em Monotone difference approximations for
  scalar conservation laws}, Math. Comp., 34 (1980), pp.~1--21.

\bibitem{goettlich2010supplychains}
{\sc C.~D'Apice, S.~G\"{o}ttlich, M.~Herty, and B.~Piccoli}, {\em Modeling,
  simulation, and optimization of supply chains}, Society for Industrial and
  Applied Mathematics (SIAM), Philadelphia, PA, 2010.
\newblock A continuous approach.

\bibitem{engquist1981one}
{\sc B.~Engquist and S.~Osher}, {\em One-sided difference approximations for
  nonlinear conservation laws}, Mathematics of Computation, 36 (1981),
  pp.~321--351.

\bibitem{eymard2000finitevolume}
{\sc R.~Eymard, T.~Gallou\"{e}t, and R.~Herbin}, {\em Finite volume methods},
  in Handbook of numerical analysis, {V}ol. {VII}, Handb. Numer. Anal., VII,
  North-Holland, Amsterdam, 2000, pp.~713--1020.

\bibitem{friedrich2021network}
{\sc J.~Friedrich, S.~G\"ottlich, and M.~Osztfalk}, {\em Network models for
  nonlocal traffic flow}, ESAIM: M2AN, 56 (2022), pp.~213--235.

\bibitem{friedrich2020nonlocal}
{\sc J.~Friedrich, S.~G\"ottlich, and E.~Rossi}, {\em Nonlocal approaches for
  multilane traffic models}, Commun. Math. Sci., 19 (2021), pp.~2291--2317.

\bibitem{friedrich2019maximum}
{\sc J.~Friedrich and O.~Kolb}, {\em Maximum principle satisfying {CWENO}
  schemes for nonlocal conservation laws}, SIAM J. Sci. Comput., 41 (2019),
  pp.~A973--A988.

\bibitem{friedrich2018godunov}
{\sc J.~Friedrich, O.~Kolb, and S.~G\"{o}ttlich}, {\em A {G}odunov type scheme
  for a class of {LWR} traffic flow models with non-local flux}, Netw. Heterog.
  Media, 13 (2018), pp.~531--547.

\bibitem{goatin2016well}
{\sc P.~Goatin and S.~Scialanga}, {\em Well-posedness and finite volume
  approximations of the lwr traffic flow model with non-local velocity},
  Networks and Heterogeneous Media, 11 (2016), pp.~107--121.

\bibitem{GodlewskiRaviart}
{\sc E.~Godlewski and P.-A. Raviart}, {\em Numerical approximation of
  hyperbolic systems of conservation laws}, vol.~118 of Applied Mathematical
  Sciences, Springer-Verlag, New York, 1996.

\bibitem{godunov1959}
{\sc S.~K. Godunov}, {\em A difference method for numerical calculation of
  discontinuous solutions of the equations of hydrodynamics}, Mat. Sb. (N.S.),
  47 (89) (1959), pp.~271--306.

\bibitem{gottlich2014modeling}
{\sc S.~G{\"o}ttlich, S.~Hoher, P.~Schindler, V.~Schleper, and A.~Verl}, {\em
  Modeling, simulation and validation of material flow on conveyor belts},
  Appl. Math. Model., 38 (2014), pp.~3295--3313.

\bibitem{holden2015front}
{\sc H.~Holden and N.~H. Risebro}, {\em Front tracking for hyperbolic
  conservation laws}, vol.~152, Springer, 2015.

\bibitem{huang2023asymptotically}
{\sc K.~Huang and Q.~Du}, {\em Asymptotically compatibility of a class of
  numerical schemes for a nonlocal traffic flow model}, arXiv preprint
  arXiv:2301.00803,  (2023).

\bibitem{KeimerPflug2017}
{\sc A.~Keimer and L.~Pflug}, {\em Existence, uniqueness and regularity results
  on nonlocal balance laws}, J. Differential Equations, 263 (2017),
  pp.~4023--4069.

\bibitem{KEIMER2023}
{\sc A.~Keimer and L.~Pflug}, {\em Nonlocal balance laws – an overview over
  recent results}, Handbook of Numerical Analysis, Elsevier, 2023.

\bibitem{keimer2018multi}
{\sc A.~Keimer, L.~Pflug, and M.~Spinola}, {\em Existence, uniqueness and
  regularity of multi-dimensional nonlocal balance laws with damping}, J. Math.
  Anal. Appl., 466 (2018), pp.~18--55.

\bibitem{keimer2018bounded}
\leavevmode\vrule height 2pt depth -1.6pt width 23pt, {\em Nonlocal scalar
  conservation laws on bounded domains and applications in traffic flow}, SIAM
  J. Math. Anal., 50 (2018), pp.~6271--6306.

\bibitem{kruvzkov1970first}
{\sc S.~N. Kru{\v{z}}kov}, {\em First order quasilinear equations with several
  independent variables.}, Mat. Sb. (N.S.), 81 (123) (1970), pp.~228--255.

\bibitem{lax1954scheme}
{\sc P.~D. Lax}, {\em Weak solutions of nonlinear hyperbolic equations and
  their numerical computation}, Comm. Pure Appl. Math., 7 (1954), pp.~159--193.

\bibitem{leveque1992numerical}
{\sc R.~J. LeVeque}, {\em Numerical methods for conservation laws}, vol.~214,
  Springer, 1992.

\bibitem{leveque2002finite}
\leavevmode\vrule height 2pt depth -1.6pt width 23pt, {\em Finite volume
  methods for hyperbolic problems}, vol.~31, Cambridge university press, 2002.

\bibitem{rossi2020well}
{\sc E.~Rossi, J.~Wei{\ss}en, P.~Goatin, and S.~G{\"o}ttlich}, {\em
  Well-posedness of a non-local model for material flow on conveyor belts},
  ESAIM: Mathematical Modelling and Numerical Analysis, 54 (2020),
  pp.~679--704.

\bibitem{thomas2013numerical}
{\sc J.~W. Thomas}, {\em Numerical partial differential equations: finite
  difference methods}, vol.~22, Springer Science \& Business Media, 2013.

\end{thebibliography}
\end{document}